\documentclass[10pt]{amsart}
\usepackage{amsmath,amssymb,amsthm}
\usepackage{graphicx}
\usepackage{color}
\newtheorem{theorem}{Theorem}[section]    
\newtheorem{lemma}[theorem]{Lemma} 
 
\newtheorem{proposition}[theorem]{Proposition}

\newtheorem{corollary}[theorem]{Corollary}
\theoremstyle{definition}
\newtheorem{definition}[theorem]{Definition}

\newtheorem{remark}[theorem]{Remark}
\newtheorem*{remark*}{Remark}
\newcommand{\Z}{\mathbb{Z}}

 \makeatletter
    
    \@addtoreset{equation}{section}
  \makeatother

\title[A clasp number two genus two fibered knot]{A partial classification of 3-dimensional clasp number two, genus two fibered knots}
\author{Tetsuya Ito}
\address{Department of Mathematics, Kyoto University, Kyoto 606-8502, JAPAN}
\email{tetitoh@math.kyoto-u.ac.jp}

\begin{document}

\begin{abstract}
The 3-dimensional clasp number $cl(K)$ of a knot $K$ is the minimum number of clasp singularities of clasp disk, a singular immersed disk bounding $K$ whose singular set consists of only clasp singularities. We give a classification of clasp number two, genus two fibered knots under the assumption that they admit a clasp disk of certain type which we call of type II.
\end{abstract}

\maketitle

\section{Introduction}

Every knot $K$ in $S^{3}$ bounds a clasp disk $D$, a singular immersed disk whose singular set consists of only clasp singularities. The \emph{3-dimensional clasp number}\footnote{There is a notion of 4-dimensional clasp number. In this paper we only treat 3-dimensional clasp number.} $cl(K)$ of a knot $K$, which we simply call the \emph{clasp number}, is the minimum of $cl(D)$ for a clasp disk $D$ of $K$, where $cl(D)$ is the number of clasp singularities of $D$\footnote{Clasp number is often denoted by $c(K)$, but to avoid confusion with crossing number we adopt to use $cl(K)$.}. 
The sign of a clasp singularity is defined as the sign of intersection of disk and $K$ at the clasp (see Figure \ref{fig:clasp} (a)).

\begin{figure}[htbp]
\begin{center}
\includegraphics*[width=100mm]{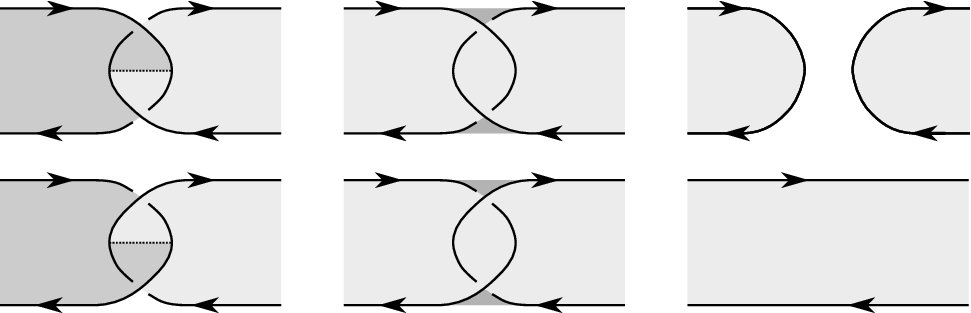}
\begin{picture}(0,0)
\put(-305,90){(a)}
\put(-205,90){(b)}
\put(-100,90){(c)}
\put(-100,40){(d)}
\end{picture}
\caption{(a) positive and negative clasp singularity. (b) The Seifert surface from clasp disk. (c) resolution of clasp singularity. (d) smoothing of clasp singularity} 
\label{fig:clasp}
\end{center}
\end{figure}
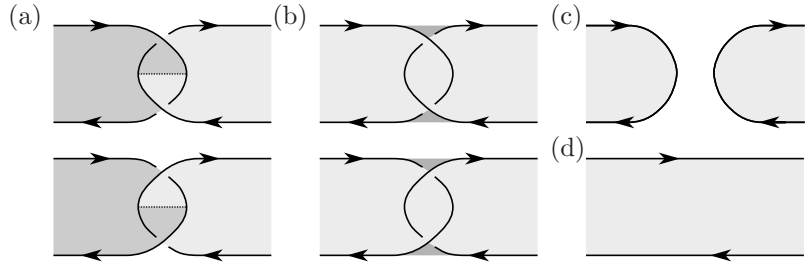

Since one clasp singularity is resolved by a single crossing change (see Figure \ref{fig:clasp} (c)), it follows that $u(K) \leq cl(K)$, where $u(K)$ is the unknotting number of $K$. 

On the other hand, a clasp disk $D$ can be converted to a Seifert surface $S_D$ by replacing a neighborhood of the clasp as in Figure \ref{fig:clasp} (b). More precisely, the Seifert surface $S_D$ is obtained by the surface $S_{D}^{0}$ by plumbing a positive or negative Hopf band, where $S_D^{0}$ is a surface obtained by replacing the clasp with untwisted band as in as in Figure \ref{fig:clasp} (d).
We call such an operation that replaces a clasp with a band the \emph{smoothing} of a clasp singularity. In particular, it follows that $g(K) \leq cl(K)$, where $g(K)$ is the genus of the knot $K$.

It is easy to see that a knot with clasp number one is exactly the doubled knots. 
Since a knot $K$ is a doubled knot if and only if $g(K)=u(K)=1$ \cite{ko,scto}, it follows that $cl(K)=1$ if and only if $g(K)=u(K)=1$. 

The situation differs when the clasp number is greater than one. In \cite{kk} Kadokami and Kawamura showed that the Conway polynomial of clasp number two knots has some special expression. As a consequence, they showed that a clasp number two knot cannot have the Conway polynomial $\nabla_K(z)$ of the form $a_4(K)z^4+a_2(K)z^2+1$ with $a_4(K) \equiv 3 \pmod 8$ and $a_2(K) \equiv 2 \pmod 4$. Here $a_2(K)$ and $a_4(K)$ denote the coefficients of $z^{2}$ and $z^{4}$ of the Conway polynomial $\nabla_K(z)$, respectively. Using this constraint, they gave infinitely many knots $\{K_n\}$ that satisfy $cl(K_n) > 2 \geq \max\{g(K_n),u(K_n)\}$. 

According to the positions of two clasps, a clasp number two clasp disk is naturally classified into two types which we call \emph{of type X} and \emph{of type II}. We say that a clasp disk is of \emph{type X} (resp. of \emph{type II}) if the  smoothing of the clasps as in Figure \ref{fig:clasp} (d) gives a genus one surface with single boundary (resp. genus zero surface with three boundaries). See Definition \ref{definition:type} for details.

The aim of this note to determine \emph{fibered} genus two knots with clasp number  two, under the assumption that they admit a clasp disk of type II.

\begin{theorem}[Classification of genus two, clasp number two fibered knots of type II]
\label{theorem:main}

Let $K$ be a genus two fibered knot with $cl(K)=2$. If $K$ admits a clasp disk of type II, then up to taking mirror image, $K$ is one the following knot.
\begin{itemize}
\item[(i)] (Non-prime knots) $3_1\# 3_1, 3_1 \# 4_1, 4_1 \# 4_1, 3_1 \# \overline{3_1}$.
\item[(ii)] (Two-bridge knots) $6_2,6_3,7_6,7_7$.
\item[(iii)] (Montesions knots) The Montesions knots
\begin{align*}
&K\left(\frac{1}{2}, -\frac{2}{3},\frac{2}{4n \pm 1}\right),
K\left(\frac{1}{2},-\frac{2}{5},\frac{2}{4n \pm 1}\right),
K\left(\frac{1}{2n}, \frac{2}{3},-\frac{2}{3}\right),\\
&K\left(\frac{1}{2n}, \frac{2}{3},-\frac{2}{5}\right),
K\left(\frac{1}{2n}, \frac{2}{5},-\frac{2}{3}\right),
K\left(\frac{1}{2n}, \frac{2}{5},-\frac{2}{5}\right) \quad (n \in \Z)
\end{align*}
\item[(iv)] (Exceptional knots) The $12$ knots $K^{\sf ex}_{i;\varepsilon_1,\varepsilon_2}$ $(i=1,2,3, \varepsilon_1,\varepsilon_2 \in \{\pm 1\})$ in Figure \ref{fig:exceptional}.
\end{itemize}
\end{theorem}

\begin{figure}[htbp]
\begin{center}
\includegraphics*[width=130mm]{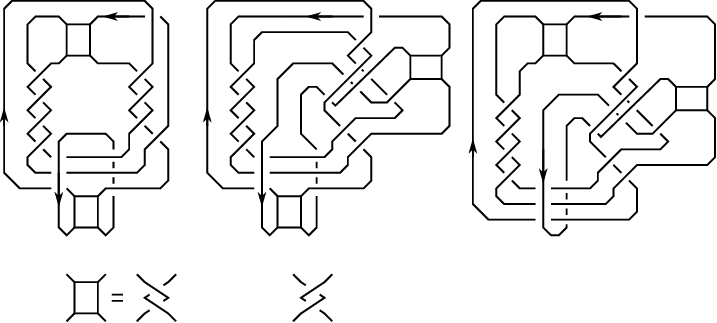}
\begin{picture}(0,0)
\put(-144,157){$\varepsilon_1$}
\put(35,147){\rotatebox{270}{$\varepsilon_1$}}
\put(172,130){\rotatebox{270}{$\varepsilon_1$}}
\put(-36,68){$\varepsilon_2$}
\put(-139,22){$\varepsilon$}
\put(-85,22){$(\varepsilon = +1)$}
\put(-38,10){$,$}
\put(3,22){$(\varepsilon = -1)$}
\put(-140,68){$\varepsilon_2$}
\put(103,157){$\varepsilon_2$}
\put(-110,60){\large $K^{\sf ex}_{1;\varepsilon_1,\varepsilon_2}$}
\put(0,60){\large $K^{\sf ex}_{2;\varepsilon_1,\varepsilon_2}$}
\put(160,60){\large $K^{\sf ex}_{3;\varepsilon_1,\varepsilon_2}$}
\end{picture}
\caption{12 exceptional knots $K^{\sf ex}_{i;\varepsilon_1,\varepsilon_2}$} 
\label{fig:exceptional}
\end{center}
\end{figure} 

For our terminologies of Montesions knots $K(r_1,r_2,r_3)$, see Appendix \ref{section:rational_tangle}.

Although the Theorem \ref{theorem:main} only treats the case $K$ admits a clasp disk of type II, we will observe several constraints for knots to admit a clasp disk of type X hence we are often able to detect knots with $cl(K)>2\geq \max\{g(K),u(K)\}$. 
 
\begin{corollary}
\label{cor:cl>2} 
The knots $K \in \{11n_{74}, 11n_{116}, 11n_{142},12n_{462},12n_{838}\}$ satisfy
$u(K)\leq g(K)=2 < 3 \leq cl(K)$. 
\end{corollary}
\begin{proof}
These knots are fibered, genus two knot with $u(K)\leq 2$.
Furthermore, these knots have even $a_2(K)$ hence by Corollary \ref{cor:Conway-obstruction-typeX} they do not admit clasp disk of type X. Since these knots do not appear in the list of Theorem \ref{theorem:main}, they do not admit a clasp disk of type II, either. 
\end{proof}

A genus two fibered knot has $a_4(K)= \pm 1 \not \equiv 3 \pmod 8$ so aforementioned Kadokami-Kawamura's obstruction does not apply. The knots in Corollary \ref{cor:cl>2} gives first examples of fibered knots that satisfy $cl(K)>\max\{g(K),u(K)\}$.

\section*{Acknowledgement}
The author is partially supported by JSPS KAKENHI Grant Number 21H04428, 23K03110 and 25H00588. He would like to thank Teruhisa Kadokami for his comments.

\section{A quick review of Conway and HOMFLY polynomial of clasp number two knots}

In this section we quickly review a computation of Conway and the zero-th coefficient HOMFLY polynomial of clasp number two knots that has been appeared in \cite{kk,ta}. We will observe these invariants give a constraint for a knot to admit a clasp disk of type X.

\subsection{Conway polynomial}
Let $K=K_{\gamma_1,\gamma_2}$ be a knot with clasp number two and let $D$ be its clasp disk with two clasp sigularities $\gamma_1,\gamma_2$. 
We denote by $\varepsilon_i$ ($i=1,2$) the sign of the clasp singularity $\gamma_i$.

We consider the following knots derived from $K$ that forms a skein resolution tree 
(See Figure \ref{fig:clasp-skein} below).

\begin{itemize}
\item $K_{u,\gamma_2}$ is the link obtained from $K$ by resolving the clasp $\gamma_1$.
\item $K_{o,\gamma_2}$ is the link obtained from $K$ by smoothing the clasp $\gamma_1$.
\item $K_{u,u} $ is the link obtained from $K_{u,\gamma_2}$ by resolving the clasp $\gamma_2$. This is just the unknot.
\item $K_{u,o}$ is the link obtained from $K_{u,\gamma_2}$ by smoothing the clasp $\gamma_2$.
\item $K_{o,u}$ is the link obtained from $K_{o,\gamma_2}$ by resolving the clasp $\gamma_2$. 
\item $K_{o,o}$ is the link obtained from $K_{o,\gamma_2}$ by smoothing the clasp $\gamma_2$.
\end{itemize}

\begin{figure}[htbp]
\begin{center}
\includegraphics*[width=120mm]{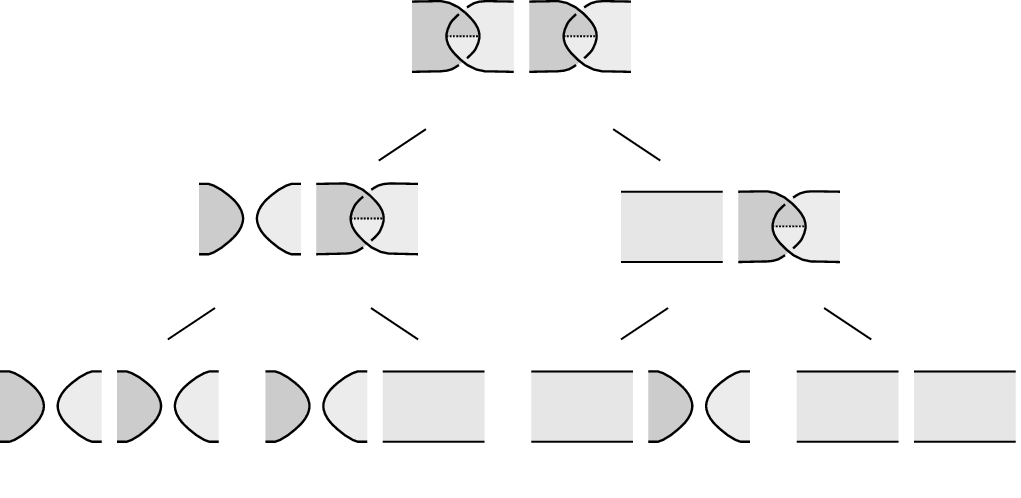}
\begin{picture}(0,0)
\put(-195,130){$K=K_{\gamma_1,\gamma_2}$}
\put(-255,68){$K_{u,\gamma_2}$}
\put(-112,68){$K_{o,\gamma_2}$}
\put(-335,5){$K_{u,u}(=\mbox{Unknot})$}
\put(-224,5){$K_{u,o}$}
\put(-133,5){$K_{o,u}$}
\put(-46,5){$K_{o,o}$}
\end{picture}
\caption{Skein tree: the knots $K_{u,\gamma_2}$, $K_{o,\gamma_2}$,$K_{u,u}$,$K_{u,o}$,$K_{o,u}$, $K_{o,o}$} 
\label{fig:clasp-skein}
\end{center}
\end{figure}
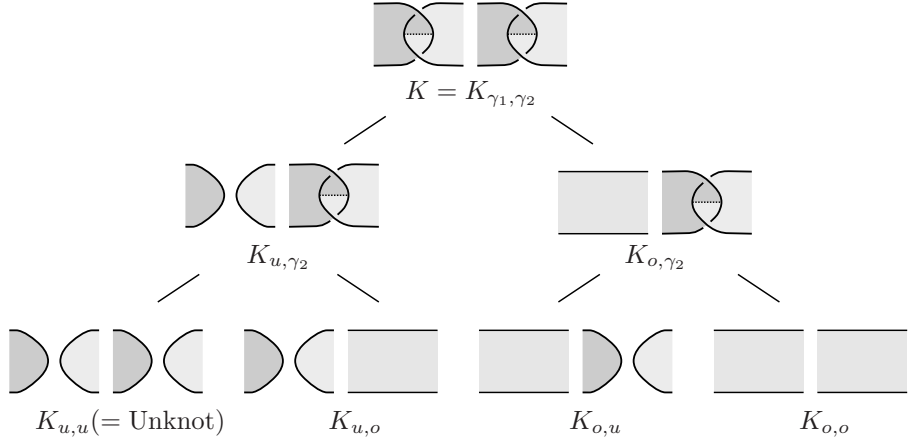 

We denote by $S_{\ast,\ast}$ ($\ast \in \{o,u\}$) the Seifert surfaces of $K_{\ast,\ast}$ that comes from the clasp disk $D$. 
In the following, we often regard $S_{o,u}$ and $S_{u,o}$ as subsurfaces of $S_{o,o}$ in an obvious manner.

By the skein relation, the Conway polynomial of $K$ is given by
\[ \nabla_{K}(z) = 1+ \varepsilon_1z\nabla_{K_{o,u}}(z)+\varepsilon_2z\nabla_{K_{u,o}}(z) + \varepsilon_1\varepsilon_2z^{2}\nabla_{K_{o,o}}(z)\]

To give more detailed expressions of the Conway polynomials of $K_{o,u},K_{u,o}, K_{o,o}$, we classify the clasp disk into the following two types. 

\begin{definition}
\label{definition:type}
We say that a clasp disk $D$ is \emph{of type $X$} (resp. \emph{of type II}) if $K_{o,o}$ is a knot (resp. 3-component link)\footnote{In \cite{ta}, type X and type II is referred as type 0 and type 1, respectively. }.
\end{definition}

See Figure \ref{fig:clasp-surface-type} for schematic illustrations of clasp disks of type X and type II, where the meaning of `X' and `II' are clear.

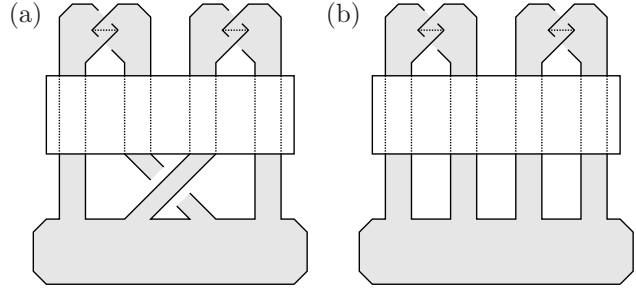
\begin{figure}[htbp]
\begin{center}
\includegraphics*[width=80mm]{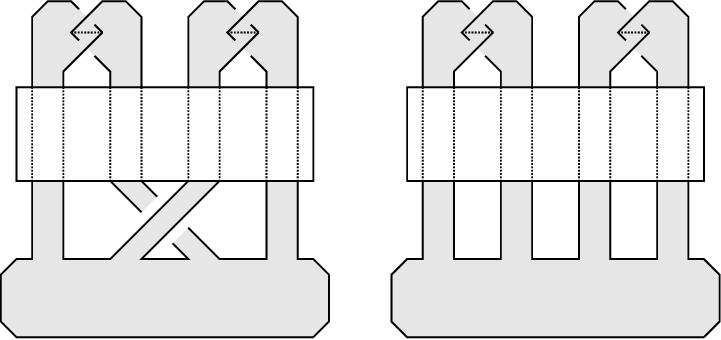}
\begin{picture}(0,0)
\put(-240,100){(a)}
\put(-120,100){(b)}
\end{picture}
\caption{(a) Clasp disk of type X. (b) Clasp disk of type II. A box represents some tangle.} 
\label{fig:clasp-surface-type}
\end{center}
\end{figure} 

First of all, $K_{o,u}$ and $K_{u,o}$ are two-component links bounding an annulus so  its Conway polynomial is given by 
\[ \nabla_{K_{o,u}}(z) = \ell_1 z, \quad \nabla_{K_{u,o}}(z) = \ell_2 z\]
where $\ell_1,\ell_2$ are the linking number of the links $K_{o,u}$ and $K_{u,o}$, respectively.

\begin{figure}[htbp]
\begin{center}
\includegraphics*[width=90mm]{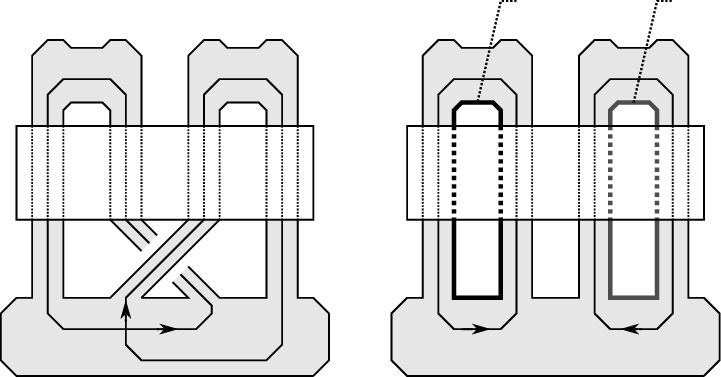}
\begin{picture}(0,0)
\put(-265,130){(i)}
\put(-250,10){$\alpha$}
\put(-155,10){$\beta$}
\put(-145,130){(ii)}
\put(-110,10){$\alpha$}
\put(-15,10){$\beta$}
\put(-70,130){$K_1$}
\put(-15,130){$K_2$}
\end{picture}
\caption{(i) Basis for type X case. (ii) Basis for type II case. A key feature is that the component $K_1$ is equal to the component of $K_{o,u}$ and the component $K_2$ is equal to $K_2$.} 
\label{fig:clasp-basis}
\end{center}
\end{figure}
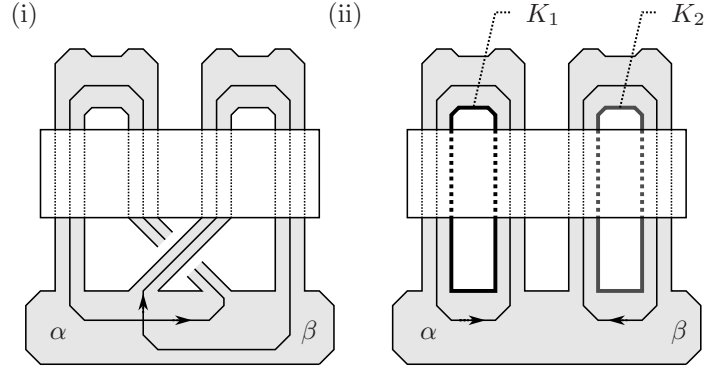 

Assume that the clasp disk $D$ is of type X.
Let $\alpha$ and $\beta$ be the simple closed curves on $S_{o,o}$ taken as follows (see Figure \ref{fig:clasp-basis} (i)).
\begin{itemize}
\item $\alpha$ is the core of the annulus $S_{o,u}$. 
\item $\beta$ is the core of the annulus $S_{u,o}$.
\item The orientation of $\alpha$ and $\beta$ are taken so that they positively intersect at one point.
\end{itemize}

Let $\alpha_+$ be the curve $\alpha$ pushed off along the normal direction of $S$, and let $\alpha_{+}^r$ be the curve $\alpha_+$ with reverse orientation.
The curves $\beta_+$ and $\beta_{+}^r$ are defined similarly.

Then
\begin{align*}
\ell_1 &= lk(K_{o,u}) = lk(\alpha,\alpha_{+}^{r}) = -lk(\alpha,\alpha_{+})\\
\ell_2 &= lk(K_{u,o}) = lk(\beta,\beta_+^r) = -lk(\beta,\beta_+^r). 
\end{align*}
Finally, let 
\[ \ell = lk(\alpha,\beta_+) \]
Thus the Seifert matrix of $S_{o,o}$ is $\begin{pmatrix}-\ell_1 & \ell \\ \ell +1 & -\ell_2 \end{pmatrix}$ hence
\[ \nabla_{K_{o,o}}(z) =1+ (\ell_1\ell_2 - \ell(\ell+1))z^{2}  \]

When the clasp disk $D$ is of type II, we take the simple closed curves $\alpha$ and $\beta$ similarly (see Figure \ref{fig:clasp-basis} (ii)). 
Then the Seifert matrix of $S_{o,o}$ is $\begin{pmatrix}-\ell_1 & \ell\\ \ell & -\ell_2 \end{pmatrix}$ hence
\[ \nabla_{K_{o,o}}(z) =1+ (\ell_1\ell_2 - \ell^2)z^{2}\]

Summarizing, we get the following formula of the Conway polynomial of clasp number two knots.

\begin{theorem}\cite{kk}
\label{theorem:Conway}
If $K$ is a clasp number two knot, then 
\[ \nabla_K(z) = \begin{cases} 1+ (\varepsilon_1\ell_1 + \varepsilon_2\ell_2+\varepsilon_1\varepsilon_2) z^{2} + \varepsilon_1 \varepsilon_2(\ell_1\ell_2-\ell(\ell+1))z^{4} & (\mbox{if }D\mbox{ is of type }X)\\
1+ (\varepsilon_1\ell_1 + \varepsilon_2\ell_2) z^{2} + \varepsilon_1 \varepsilon_2(\ell_1\ell_2-\ell^2)z^{4} & (\mbox{if }D\mbox{ is of type }II) \end{cases}\]
\end{theorem}

As an application of the Conway polynomial computation, we give the following simple constraint for knots to admit a clasp disk of type X.

\begin{corollary}
\label{cor:Conway-obstruction-typeX}
If $K$ is a clasp number two genus two knot. If $a_4(K)\equiv 1 \pmod{2}$ and $a_{2}(K) \equiv 0 \pmod{2}$ then $K$ does not admit a clasp disk of type X.
\end{corollary}
\begin{proof}
Assume that $K$ admits a clasp disk of type X.
By Theorem \ref{theorem:Conway}
$a_4(K) =\varepsilon_1 \varepsilon_2(\ell_1\ell_2-\ell(\ell+1)) \equiv \ell_1\ell_2 \pmod{2}$ hence if $a_4(K) \equiv 1 \pmod{2}$ then $a_2(K)=\varepsilon_1\ell_1+  \varepsilon_2\ell_2 \equiv \ell_1 + \ell_2 \equiv 0 \pmod{2}$.  
\end{proof}

In particular, it follows that genus two, clasp number two fibered knot $K$ does not admit a clasp disk of type X if $a_2(K)$ is even.

\begin{remark}
\label{remark:ell}
When $D$ is of type II, the component $K_1$ of $K_{o,o}$ which is parallel to $\alpha^{r}$ ($\alpha$ with reverse orientation) is identical with the component of the link $K_{o,u}$.
Similarly, the component $K_2$ of $K_{o,o}$ which is parallel to $\beta$ is identical with the component of the link $K_{u,0}$. The remaining component $K_3$ is a band sum of $K_1$ and $K_2$. Therefore their linking numbers are given by 
\begin{equation}
\label{eqn:ell-formula}
\begin{cases}
lk(K_1,K_2)=lk(\alpha^{r},\beta) = -lk(\alpha,\beta_{+}) = -\ell \\
lk(K_1,K_3)=lk(\alpha^{r}, \alpha_{+}) + lk(\alpha^{r}, \beta^{r}) =\ell_1 + \ell \\
lk(K_2,K_3) = lk(\beta, \alpha) + lk(\beta,\beta^{r}) = \ell_2 + \ell \\
\end{cases}
\end{equation}
Thus when $D$ is of type II, $\nabla_K(z)$ is determined by the signs $\varepsilon_1,\varepsilon_2$ of the clasps and the linking numbers of the link $K_{o,o}$.
\end{remark}

\subsection{The (zero-th) coefficient of the HOMFLY polynomial}

We use the same argument to give a formula of the zeroth-coefficient of HOMFLY polynomial.

The HOMFLY polynomial of a knot or link $K$ is defined by the skein relation
\[ v^{-1} P_{K_+}(v,z) -  vP_{K_-}(v,z)= z P_{K_0}(v,z), \quad P_{\sf Unknot}(v,z) =1 \]
where $K_+, K_-,K_0$ are so-called the skein triple, links that are identical except in a small ball in which they are as illustrated below.
\[ \begin{picture}(30,45)
\put(0,15){\vector(1,1){30}}
\put(30,15){\line(-1,1){10}}
\put(10,35){\vector(-1,1){10}}
\put(8,0){$K_+$}
\end{picture}
\qquad
\begin{picture}(30,45)
\put(0,15){\line(1,1){10}}
\put(30,15){\vector(-1,1){30}}
\put(20,35){\vector(1,1){10}}
\put(8,0){$K_-$}
\end{picture} 
\qquad
\begin{picture}(30,45)
\qbezier(0, 15)(20,30)(0,45)
\put(0,45){\vector(-1,1){0}}
\qbezier(30, 15)(10,30)(30,45)
\put(30,45){\vector(1,1){0}}
\put(10,0){$K_0$}
\end{picture} 
\]
It is known that $(v^{-1}z)^{\#K-1}P_K(v,z) \in \Z[v^{-2},v^2][z^{2}]$ so we write
\[ P_K(v,z)=(v^{-1}z)^{-\#K + 1} \sum_{i=0} p^{i}_K(v)z^{2i}, \qquad p^{i}_K(v) \in \Z[v^2,v^{-2}] \]
where $\#K$ is the number of components of $K$. We call the polynomial $p^{i}_K(v)$ the \emph{$i$-th coefficient (HOMFLY) polynomial} of $K$.

The skein relation of the HOMFLY polynomial leads to the following much simpler skein relation 
\[ v^{-2}p_{K_+}^{0}(v) - p_{K_-}^0(v)= \begin{cases} p^{0}_{K_0}(v) & \mbox{ if } \delta=0,  \\
0 & \mbox{ if } \delta=1. \end{cases}\]
of the zeroth coefficient polynomial $p^{0}_K(v)$.
Here $\delta =\frac{1}{2}(\#K_{+}-\# K_{0} +1)$. Namely, $\delta = 0$ (resp. $1$) if and only if the two strands of the skein crossing belong to the same component  (resp. different components).

Consequently, for an $n$-component link $K= K_1 \cup \cdots\cup K_n$,  the zeroth coefficient polynomial of link is determined by the zeroth coefficient polynomials $p^{0}_{K_i}(v)$ of its components and the (total) linking number $lk(K)=\sum_{i<j}lk(K_i,K_j)$, as
\[ 
p^0_{K}(v) =(v^{-2}-1)^{n-1} v^{2lk(K)}p^{0}_{K_1}(v)p^{0}_{K_2}(v) \cdots p^{0}_{K_n}(v)
\]
When a 2-component link $L=K\cup K'$ bounds an annulus, then their two components are the same knot with opposite orientation. Since the HOMFLY polynomial of knots are insensitive to orientation, it follows that 
\begin{equation}
\label{eqn:p_0-annulus}
p^0_{L}(v) =(v^{-2}-1)v^{2lk(K)}p^{0}_{K}(v)^2 = (v^{-2}-1)v^{2lk(K)}p^{0}_{K'}(v)^2. 
\end{equation} 

\begin{remark}
These particular properties of $p^{0}_K(v)$ have applications for cosmetic crossing conjecture \cite{it1} and unknotting number \cite{it2}.
Also, the simplicity of the skein relation implies that the computation of zeroth coefficient polynomial is much easier than the computation of the whole HOMFLY polynomial, both in theoretically \cite{pr} and in practice.
It can serve as a tool to distinguish infinite family of knots which are complicated enough to compute other invariants (see, for example, \cite{ta2,kp}).
\end{remark}

One can compute the zeroth coefficient polynomial of clasp number two knot $K$ in a similar manner as Conway polynomials. By the skein relation, the HOMFLY polynomial of $K$ is given by
\begin{align}
\label{eqn:P_K-formula}
P_{K}(v,z) & = v^{2\varepsilon_1+2\varepsilon_2} + \varepsilon_1 v^{\varepsilon_1+2\varepsilon_2} zP_{K_{o,u}}(v,z) \\
& \qquad +\varepsilon_2 v^{2\varepsilon_1+\varepsilon_2} zP_{K_{u,o}}(v,z) + \varepsilon_1\varepsilon_2 v^{\varepsilon_1+\varepsilon_2}z^{2}P_{K_{o,o}}(v,z). \nonumber
\end{align}
Since both $K_{o,u}$ and $K_{u,o}$ are two component links bounding an annulus, by  
\eqref{eqn:p_0-annulus}
\[ p^0_{K_{o,u}}(v) = (v^{-2}-1)v^{2\ell_1}p^{0}_{K_1}(v)^2 \]
\[ p^0_{K_{u,o}}(v) = (v^{-2}-1)v^{2\ell_1}p^{0}_{K_2}(v)^2 \]
where $K_1$ (resp. $K_2$) is a component of $K_{o,u}$ (resp. $K_{u,o}$).

When $D$ is of type X, the knot $K_{o,o}$ does not contribute to $p^{0}_K(v)$.

When $D$ is of type II, $K_{o,o}$ is a 3-component link whose components are $K_1,K_2$ and the other component $K_3$. By \eqref{eqn:ell-formula} it follows that 
\begin{align*}
P_{K_{o,o}}(v,z) &= (v^{-1}-v)^2v^{2(lk(K_1,K_2)+lk(K_2,K_3)+lk(K_1,K_3))}p^{0}_{K_1}(v)p^{0}_{K_2}(v)p^{0}_{K_3}(v) \\
&= (v^{-1}-v)^2v^{2(\ell_1+\ell_2+\ell)}p^{0}_{K_1}(v)p^{0}_{K_2}(v)p^{0}_{K_3}(v)
\end{align*}

Therefore by \eqref{eqn:P_K-formula} we get a more precise version of the formula of the zeroth coefficient.
\begin{theorem}\cite[Theorem 1.3]{ta}
\label{theorem:HOMFLY0}
Let $K$ be a clasp number two knot. 
\begin{itemize}
\item[(i)] If $K$ admits a clasp disk of type X,
\begin{align*}
\label{eqn:p0-typeX}
p^{0}_K(z) &= v^{2(\varepsilon_1+\varepsilon_2)}  
+ \varepsilon_1 v^{\varepsilon_1+2\varepsilon_2}(v^{-1}-v)v^{2\ell_1}(p^{0}_{K_1}(v))^{2} \\
& \qquad + \varepsilon_2 v^{\varepsilon_2+2\varepsilon_1}(v^{-1}-v)v^{2\ell_2}(p^{0}_{K_2}(v))^{2}
\end{align*}
where $\varepsilon_1,\varepsilon_2 \in \{ \pm 1\}, \ell_1, \ell_2 \in \Z$, such that $a_2(K) = \varepsilon_1\ell_1 + \varepsilon_1\ell_2+\varepsilon_1\varepsilon_2$.
\item[(ii)] If $K$ admits a clasp disk of type II, then
\begin{align*}
p^{0}_K(v) &= v^{2(\varepsilon_1+\varepsilon_2)}  
+ \varepsilon_1 v^{\varepsilon_1+2\varepsilon_2}(v^{-1}-v)v^{2\ell_1}(p^{0}_{K_1}(v))^{2} \\
& \qquad + \varepsilon_2 v^{\varepsilon_2+2\varepsilon_1}(v^{-1}-v)v^{2\ell_2}(p^{0}_{K_2}(v))^{2}\\
&\qquad +  \varepsilon_1\varepsilon_2v^{2(\ell_1+\ell_2+\ell)+\varepsilon_1+\varepsilon_2}(v^{-1}-v)^2p_{K_1}^{0}(v)p_{K_2}^{0}(v)p_{K_3}^{0}(v)
\end{align*} where $\varepsilon_1,\varepsilon_2 \in \{ \pm 1\}, \ell_1, \ell_2, \ell \in \Z$ such that
 $a_2(K)= \varepsilon_1\ell_1 + \varepsilon_2 \ell_2$ and  $a_4(K)=\varepsilon_1\varepsilon_2(\ell_1\ell_2-\ell^2)$.
\end{itemize}
\end{theorem}

This gives an additional constraint for a knot to admit a clasp disk of type X.
\begin{corollary}
If a clasp number two knot $K$ admits a clasp disk of type X, then
\begin{align*}
\frac{p^{0}_K(z)-v^{2(\varepsilon_1+\varepsilon_2)}}{{v^{-2}-1}}
& = \varepsilon_1 v^{1+\varepsilon_1+2(\ell_1+\varepsilon_2)}p^{0}_{K_1}(v)^2 + \varepsilon_2v^{1+\varepsilon_2+2(\ell_2+\varepsilon_1)}p^{0}_{K_2}(v)^2\\
& = \varepsilon_1 f_1(v)^2 + \varepsilon_2 f_2(v)^2
\end{align*}
for some $f_1(v),f_2(v) \in \Z[v^2,v^{-2}] \cup v\Z[v^{2},v^{-2}]$.
\end{corollary}

\section{Classification}

Our proof of classification is based on the observation that a Seifert surface from a clasp disk is obtained by Hopf plumbing.

Let $H_+$ and $H_-$ be the positive and negative Hopf link.
It is well known that they are fibered links. By \emph{positive and negative Hopf bands}, we mean the fiber surface of $H_+$ and $H_-$. By abuse of notation, we often confuse with Hopf links and Hopf bands. We use the same symbols $H_+$ and $H_-$ to represent Hopf bands rather than the Hopf links.

\begin{definition}[Hopf plumbing]
Let $K$ be a knot or link and $S$ be its connected Seifert surface. Let $\alpha$ be a properly embedded arc in $S$ and let $\gamma \subset H_{\pm}$ be the co-core of $H_{\pm}$, the properly embedded arc that its connects different boundary components. 

\emph{The positive Hopf plumbing along $\alpha$} is the surface $S_+ =S\cup_{f} H_{+}$, where $f: \alpha \times[-1,1] \rightarrow \gamma \times [-1,1]$ is a homeomorphism that identifies a regular neighborhood $\alpha \times [-1,1] \subset S$ and a regular neihborhood $\gamma \times [-1,1] \subset H_+$ so that $f(\alpha \times \{t\}) = [-1,1] \times\{p\})$ for $\alpha \in [-1,1], p \in \gamma$ and vice versa (see Figure \ref{fig:hopf-plumb}).

We say that $S_+$ is obtained from $S$ by plumbing positive Hopf band along the attaching arc $\alpha$. We often simply say that the link $L=\partial S_+$ is obtained by (positive or negative) Hopf plumbing of $K$, without referring surfaces $S$ and $S_{+}$.

The negative Hopf plumbing $S_-$ is defined similarly.
\end{definition}

\begin{figure}[htbp]
\begin{center}
\includegraphics*[width=60mm]{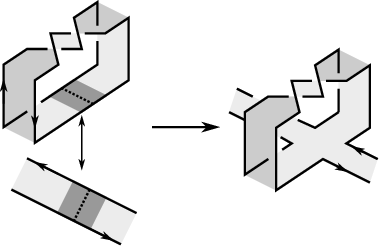}
\begin{picture}(0,0)
\put(-170,5){$S$}
\put(-170,100){$H_+$}
\put(-130,25){$\alpha$}
\put(-130,55){$\gamma$}
\put(-40,10){$S_+$}
\end{picture}
\caption{Hopf plumbing} 
\label{fig:hopf-plumb}
\end{center}
\end{figure}
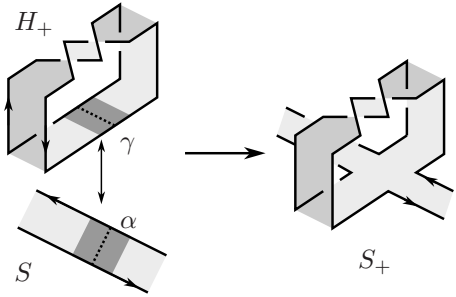 

The Hopf plumbing is a special case of more general operation called the Murasugi sum. A key property of Hopf plumbing (that also holds for Murasugi sum) which we use is the following. 

\begin{theorem}\cite{st,ga}
\label{theorem:plumbing}
Let $S_{\pm}$ be a surface obtained from $S$ by positive or negative Hopf plumbing. $S_{\pm}$ is a fiber surface (i.e. it is a fiber of the fibered link $\partial S_{\pm}$) if and only if $S$ is a fiber surface.  
\end{theorem}

Let $K$ be a genus two fibered knot of clasp number two and $D$ be its clasp disk.
Since as we have mentioned in introduction the genus two Seifert surface $S_{D}=S_{\gamma_1,\gamma_2}$ of $K=K_{\gamma_1, \gamma_2}$ from the clasp disk $D$ is obtained from the Seifert surface $S_{o,o}$ of $K_{o,o}$ by Hopf plumbing (see Figure \ref{fig:clasp} again), by Theorem \ref{theorem:plumbing} we get the following key observation.

\begin{lemma}
\label{lemma:key-observation}
Let $K$ be a genus two fibered knot of clasp number two and $D$ be its clasp disk. Then $K_{o,o}$ is either a genus one fibered knot, or, a genus zero 3-component fibered link.
The Seifert surface $S_{o,o}$ from the clasp disk $D$ is its fiber surface, and the knot $K$ is obtained from $S_{o,o}$ by plumbing two Hopf bands simultaneously.
\end{lemma}

It is well-known that genus one fibered knot is either the (positive or negative) trefoil, or, the figure-eight knot. Similarly, a genus zero three component fibered links are classified in \cite{ro}.

\begin{theorem}[Classification of genus zero, 3-component fibered link \cite{ro}]
\label{theorem:chi-1-fibered}
Let $L$ be a fibered 3-component link of genus zero. Then $L$ is either
\begin{itemize}
\item[(A)] the connected sum of two Hopf links $H_{\pm} \# H_{\pm}$.
\item[(B)] the pretzel link $P(2,2n,-2)$ for $n \in \Z$ (see Figure \ref{fig:genus-zero-fiber} (i)).
\item[(C)] the exceptional link $L^{\sf ex}$ in Figure \ref{fig:genus-zero-fiber} (ii), or its mirror image.
\end{itemize}
\end{theorem}

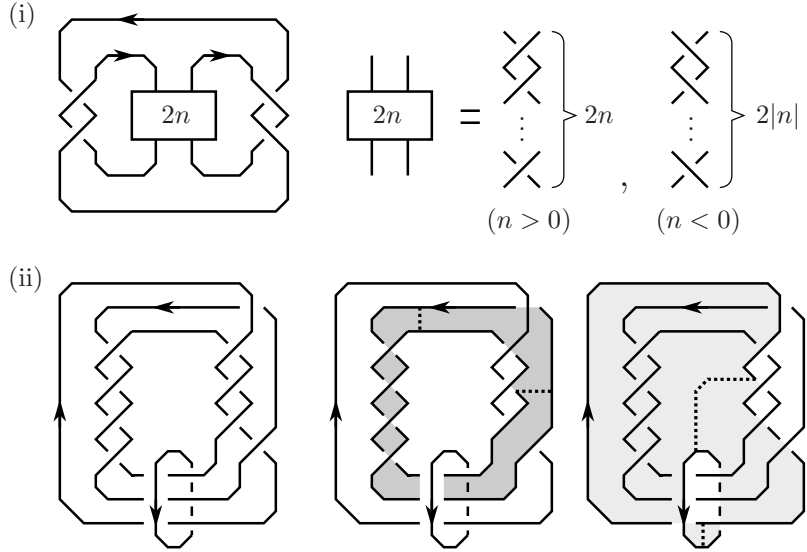
\begin{figure}[htbp]
\begin{center}
\includegraphics*[width=100mm]{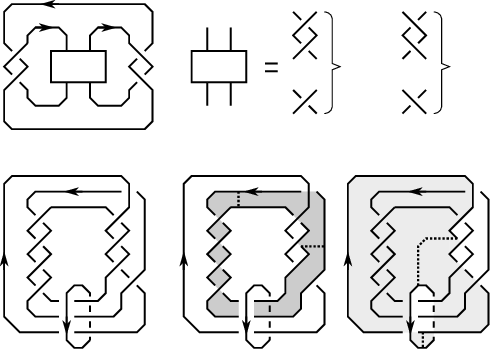}
\begin{picture}(0,0)
\put(-305,200){(i)}
\put(-305,100){(ii)}
\put(-247,162){$2n$}
\put(-167,162){$2n$}
\put(-87,162){$2n$}
\put(-124,122){$(n>0)$}
\put(-112,155){$\vdots$}
\put(-48,155){$\vdots$}
\put(-60,122){$(n<0)$}
\put(-74,138){\LARGE ,}
\put(-22,162){$2|n|$}
\end{picture}
\caption{(i) Pretzel link $P(-2,2n,2)$. (ii) Exceptional link $L^{\sf ex}$ and its fiber surface (illustrated as a union of two annuli) and its three possible attaching arcs.}
\label{fig:genus-zero-fiber}
\end{center}
\end{figure} 

Here by a pretzel link $P(2,2n,-2)$, we mean an oriented link as shown in Figure Figure \ref{fig:genus-zero-fiber} (i). In particular, $P(-2,0,2)=H_+\# H_-$.

Since our result is based on this classification theorem, we attach the proof of theorem \ref{theorem:chi-1-fibered} in Appendix \ref{sec:proof-classification} for reader's convenience and completeness.

\begin{proof}[Proof of Theorem \ref{theorem:main}]

By Lemma \ref{lemma:key-observation}, the knot $K$ is obtained by plumbing two Hopf bands along disjoint attaching arcs.
 
Since the result of Hopf plumbing is a knot, each attaching arc must connect two different components of $K_{o,o}=\partial S_{o,o}$. Up to isotopy, there are only three such arcs for three-holed sphere (see Figure \ref{fig:genus-zero-fiber} and Figure \ref{fig:pants-arc}). Furthermore two attaching arcs cannot be parallel because otherwise the plumbings yield a link.

\begin{figure}[htbp]
\begin{center}
\includegraphics*[width=40mm]{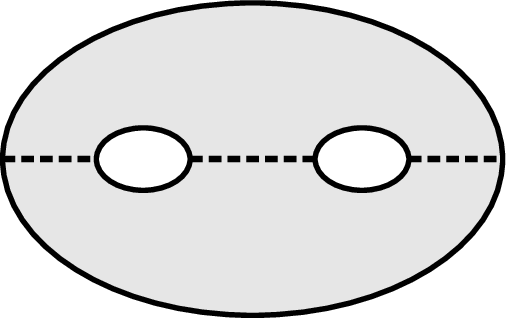}
\begin{picture}(0,0)
\end{picture}
\caption{Attaching arcs of Hopf plumbing for three-holed sphere} 
\label{fig:pants-arc}
\end{center}
\end{figure} 

This means that, each link $K_{o,o}$ in Theorem \ref{theorem:chi-1-fibered} can produce at most $2^{2}\times  \binom{3}{2}=12$ ($=$ (choice of signs of Hopf bands) $\times$ (choice of attaching arcs)) genus two, clasp number two fibered knots.

It remains to specify the knots that appears by Hopf plubmings on the genus zero 3-component fibered links $K_{o,o}$ in  Theorem \ref{theorem:chi-1-fibered}. By taking mirror image, we can assume that $K_{o,o}$ is either $H_{+}\# H_{+}$, $P(2,2n,-2)$, or, $L^{\sf ex}$.

It is convenient to use rational tangles and Montesinos knots to express the result for the case $H_{+} \# H_+$ or $P(2,2n,-2)$. For the conventions and notations of rational tangle and Montesions knot, see Appendix \ref{section:rational_tangle}.
If a positive or negative Hopf bands is attached to the portion of the $2m$-twisted bands of the pretzel knot (i.e., rational tangle $\frac{1}{2m}$), it gives rise to a rational tangle
\[ Q\left(\frac{1}{2m \pm \frac{1}{2}} \right) = Q\left(\frac{2}{4m \pm 1} \right)\]

Therefore for the case $K_{o,o}= K_{o,o}=H_{+}\# H_{+} = P(-2,0,-2)=K(-\frac{1}{2},\infty, -\frac{1}{2})$, plumbings produce one of the following knots.
\begin{itemize}
\item $K(-\frac{2}{3}, \infty, -\frac{2}{3}) = $Connected sum of two positive trefoils,
\item $K(-\frac{2}{3}, \infty, -\frac{2}{5}) = $Connected sum of the positive trefoil and the figure-eight knot.
\item  $K(-\frac{2}{3}, \infty, -\frac{2}{5}) = $Connected sum of two figure-eight knots.
\item $K(-\frac{2}{3}, \pm 2, \frac{1}{2}) = 6_2 , 6_3$
\item $K(-\frac{2}{5}, \pm 2, \frac{1}{2}) = 7_7,\overline{7_6}$
\item $K(-\frac{1}{2}, \pm 2, \frac{2}{3}) = \overline{6_2}, \overline{6_3}$
\item $K(-\frac{1}{2}, \pm 2, \frac{2}{5}) = \overline{7_7},7_6$
\end{itemize}
where $\overline{K}$ means the mirror image of $K$.

Similarly, for the case $K_{o,o}=P(2,2n,-2)$, plumbings produce one of the following knots.

\begin{itemize}
\item  $K(\frac{1}{2},\frac{2}{4n \pm 1}, -\frac{2}{3})$
\item $K(\frac{1}{2},\frac{2}{4n \pm 1},-\frac{2}{5})$
\item $K(\frac{2}{3},\frac{1}{2n}, -\frac{2}{3})$
\item  $K(\frac{2}{3},\frac{1}{2n},-\frac{2}{5})$
\item  $K(\frac{2}{5},\frac{1}{2n}, -\frac{2}{3})$
\item $K(\frac{2}{5},\frac{1}{2n},-\frac{2}{5})$
\item  $K(\frac{2}{3},\frac{2}{4n \pm 1}, -\frac{1}{2})$
\item $K(\frac{2}{5},\frac{2}{4n \pm 1},-\frac{1}{2})$
\end{itemize}

Finally, for the  case that $K_{o,o}$ is the exceptional link $L^{\sf ex}$, we get  knots $K^{\sf ex}_{i;\varepsilon_1,\varepsilon_2}$ $(i=1,2,3)$ in Figure \ref{fig:exceptional}.
\end{proof}

\appendix
\section{Classification of three component, genus zero fibered links in $S^{3}$}
\label{sec:proof-classification}
We give a proof of Theorem \ref{theorem:chi-1-fibered}. 
The proof given here is almost identical with that in \cite{ro}, except we use the fundamental group to determine the possible monodromies which substantially simplifies the proof.

For the classification, it is convenient to use terminologies of open books. An open book $(S,\phi)$ is a pair consisting of an oriented compact surface $S$ with non-empty boundary and an element $\phi$ of its mapping class group $MCG(S)$.

Let 
\[ M_{\phi} = S\times[0,1] \slash (x,1) \sim (\phi(x),0) \]
be the mapping torus of $\phi$.
Strictly speaking, the $\phi$ in the definition of $M_{\phi}$ should be understood as a homeomorphism representing $\phi$. In the following, by abuse of notation, we do not distinguish an element $\phi$ of mapping class group and a homeomorphism that represents $\phi$.

Let $M_{(S,\phi)}$ be an oriented closed 3-manifold given by
\[ M_{(S,\phi)} = M_{\phi} \cup \left( \bigsqcup_{\# \partial S} S^{1} \times D^{2} \right). \]
Here $\# \partial S$ is the number of connected components of $\partial S$, and each solid torus $S^{1} \times D^{2}$ is attached to a boundary of the mapping torus $M_{\phi}$ so that the $\{p\} \times S^{1} \subset \partial M_{\phi}$ bounds a disk in the attached solid torus $S^{1} \times D^{2}$ for all $p \in \partial S$.
We say that $(S,\phi)$ is an \emph{open book decomposition} of a 3-manifold $M$ if $M=M_{(S,\phi)}$.

We denote the the boundaries of $S$ by $C_1,C_2$ and $C_3$, and we denote by $T_i$ ($i=1,2,3$) the right-handed Dehn twist along $C_i$ (see Figure \ref{fig:pants}). The mapping class group of $S$ is a free abelian group of rank three generated by $T_1, T_2, T_3$ so we put $\phi = T_1^a T_2^b T_3^c$ $(a,b,c \in \Z)$.

Since the role of $C_1,C_2,C_3$ are interchangeable, in the following we assume that $|a| \leq |b| \leq |c|$.

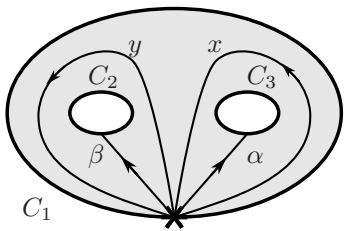
\begin{figure}[htbp]
\begin{center}
\includegraphics*[width=45mm]{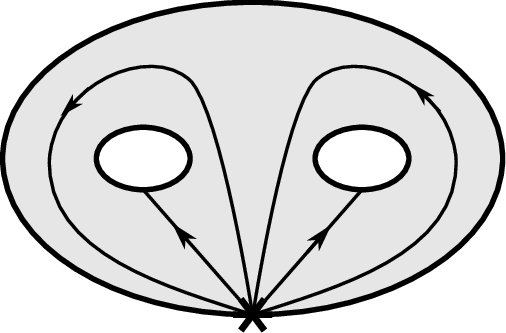}
\begin{picture}(0,0)
\put(-125,5){$C_1$}
\put(-100,55){$C_2$}
\put(-40,55){$C_3$}
\put(-100,25){$\beta$}
\put(-40,25){$\alpha$}
\put(-85,67){$y$}
\put(-55,67){$x$}
\end{picture}
\caption{Three-holed sphere $S$}
\label{fig:pants}
\end{center}
\end{figure} 

Since an open book decomposition of $S^{3}$ corresponds to a fibered link in $S^{3}$, we first classify the open book decomposition $(S,\phi)$ of the 3-sphere whose page $S$ is a genus zero surface with three boundaries.

\begin{proposition}
\label{prop:classification}
Assume that $|a| \leq |b| \leq |c|$. 
Then $\pi_1(M_{(S,(S,T_1^{a}T_2^{b}T_3^{c}))})=\{1\}$ if and only if 
\[ (a,b,c) \in \left\{\begin{array}{l}
(0,-1,-1), (0,-1,1), (0,1,-1), (0,1,1)\\
(-1,1,n),(1,-1,n) \ \quad (n \in \Z), \\
(-1,2,3),(1,-2,-3)
\end{array}\right\}
\]
\end{proposition}
\begin{proof}
We take a base point $\ast$ on $C_1$, loops $x,y$, and arcs $\alpha, \beta$ as in Figure \ref{fig:pants} (b).
By van-Kampen theorem 
\[ \pi_1(M_{(S,T_1^{a}T_2^{b}T_3^{c})}) = \langle x,y \: | \:  \overline{\alpha}\cdot \phi(\alpha), \overline{\beta} \cdot \phi(\beta) \rangle \]
Here $\cdot$ represents a concatenation of path and $\overline{\alpha}, \overline{\beta}$ are the reverse path of $\alpha$ and $\beta$, respectively.
Therefore
\[ \pi_1(M_{(S,T_1^{a}T_2^{b}T_3^{c})}) = \langle x,y \: | \: (xy)^ax^{b}=1, (xy)^{a}y^{c} = 1 \rangle\]
By putting $z=(xy)^{-1}$, we get
\[ \pi_1(M_{(S,T_1^{a}T_2^{b}T_3^{c})}) = \langle x,y,z \: | \: z^a=x^b=y^c, xyz=1 \rangle\]
hence $\pi_1(M_{(S,T_1^{a}T_2^{b}T_3^{c})})$ surjects to the triangle group 
\[ \Delta_{a,b,c} = \langle x,y,z \: | \: z^a=x^b=y^c=1, xyz=1 \rangle. \]
Since the triangle group $\Delta_{a,b,c}$ is non-trivial whenever $|a|,|b|,|c|>1$,
 we get $|a|\leq 1$.

If $a=0$, 
\[\pi_1(M_{(S,T_2^{b}T_3^{c})}) =  \langle x,y \: | \: x^{b}=1, y^{c} = 1 \rangle \] hence it is trivial if and only if $|b|=|c|=1$.

If $a=-1$, then
\begin{align*}
\pi_1(M_{(S,T_1^{-1}T_2^{b}T_3^{c})}) 
&= \langle x,y \: | \:  y^{-1}x^{b-1}=1, x^{-1}y^{c-1} = 1\rangle = \langle x\: | \:  x^{bc-b-c} =1 \rangle. 
\end{align*}
Therefore $\pi_1(M_{(S,T_1^{-1}T_2^{b}T_3^{c})})$ is trivial if and only if $bc-b-c = \pm 1$. Since we are assuming $|a| \leq |b| \leq |c|$, we conclude 
$(b,c)=(1,n)$ or $(2,3)$.

The case $a=1$ is similar.
\end{proof}

It remains to specify a fibered link that corresponds to the open books in Proposition \ref{prop:classification}. To construct a fibered link explicitly, we use \emph{Stallings twist} \cite{st} in a slightly general form following Harer \cite{ha}.

\begin{definition}[Stallings twist]
Let $S \subset S^{3}$ be a fiber surface of a fibered link $L$ in $S^{3}$ and let $\phi$ be its monodromy.
Assume that there is an essential simple closed curve on $F$ which is unknotted in $S^{3}$. Let $f$ be the framing of $c$ given by the fiber surface $S$ (i.e., the linking number of $c$ and its push-off along the normal direction of $S$).
Assume that $f+\delta_1 = \delta_2$ where $\delta_1,\delta_2 \in \{\pm 1\}$.
Then the $\delta_2$ Dehn surgery on $c$ gives a fibered link with monodromy $\phi T_c^{\delta_1}$. We call this operation \emph{Stallings twist along $c$}.
\end{definition}

When $f=0$, instead of $\delta_2 = \pm 1$ surgeries, we can apply $\frac{1}{n}$ surgery to get a fibered link with monodromy $\phi T_{c}^{-n}$. This can be understood as an iteration of the Stallings twists along the same curve $c$ on a fiber, since the framing $f=0$  means that after Stallings twist, the curve $c$ remains to have framing zero with respect to the new fiber surface.

\begin{proof}[Proof of Theorem \ref{theorem:chi-1-fibered}]
In the following, we give a fibered link with modromies in Proposition \ref{prop:classification} under appropriate numbering of its boundary components.

\begin{itemize}
\item[(i)] $T_2^{\pm 1}T_3^{\pm 1}$\\

This is the monodromy of the connected sum of positive or negative Hopf links $H_{\pm} \# H_{\pm}$.\\

\item[(ii)] $T_{1}T_2^{-1}T_{3}^{n}$\\

This is obtained by applying Stallings twists for the fiber surface $S$ of $H_+\# H_- = P(2,0,-2)$, along the curve $c$ depicted in Figure \ref{fig:s-twist} (i), namely, by applying $\frac{1}{n}$ surgery on $c$. The result is the pretzel link $P(2,-2n,-2)$, \\

\item[(iii)] $T_{1}^{-1}T_2^{2}T_3^3$\\

This is obtained by applying Stallings twist for the fiber surface $S$ of $P(2,-6,-2)$, along the curve $c$ depicted in Figure \ref{fig:s-twist} (ii). The result is the exceptional link $L^{\sf ex}$.

\begin{figure}[htbp]
\begin{center}
\includegraphics*[width=95mm]{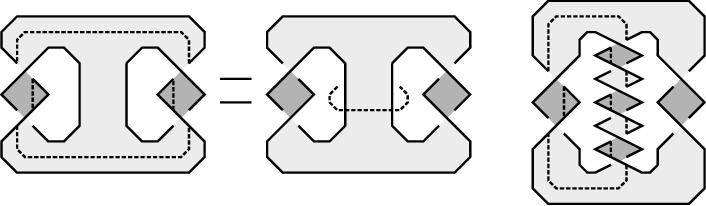}
\begin{picture}(0,0)
\put(-285,80){(i)}
\put(-238,57){$c$}
\put(-150,48){$c$}
\put(-85,80){(ii)}
\put(-60,58){$c$}
\end{picture}
\caption{(i) Stalling twists along a curve $c$ on a fiber of $H_+ \# H_- = P(2,0,-2)$. (ii) Stalling twist that yields exceptional link $L_*$.}
\label{fig:s-twist}
\end{center}
\end{figure}
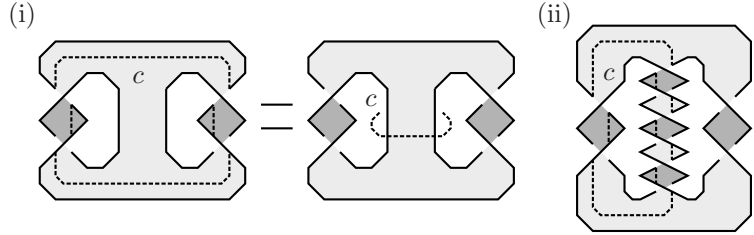 

Although it is cumbersome to draw the result of Stalling twist in the last case, there are several ways to check this, once a diagram of $L^{\sf ex}$ is given. By computing the HOMFLY polynomials, we can  easily check that $L^{\sf ex}$ is not equal to $P(2,2n,-2)$ or $H_{\pm} \# H_{\pm}$. Thus by Proposition \ref{prop:classification}, to confirm that the result of the Stalling twist coincides with $L^{\sf ex}$ (or its mirror image), it is sufficient to check that the link $L^{\sf ex}$ is indeed a genus zero fibered link. This can be checked by several ways.
\end{itemize}

The rest of the cases are obtained by taking mirror image.

\end{proof}

\section{Rational tangle and Montesinos knot}
\label{section:rational_tangle}

We review the definition of rational tangle and Montsinos knots to summarize the notation and convention. For details, we refer to \cite[Section 12]{bz}.

For an integer $n$, we define the tangles $[n],[\frac{1}{n}]$ and $[\infty]=[\frac{1}{0}]$ as in Figure \ref{fig:rational_tangle} (i).
The tangle sum $A+B$ is a tangle obtained by connecting the northeast endpoint of $A$ and northwest endpoint of $B$, and, the southeast endpoint of $A$ and southwest endpoint of $B$. Similarly, the tangle product $A\ast B$ is a tangle obtained by connecting the southwest endpoint of $A$ and northwest endpoint of $B$, and, the southeast endpoint of $A$ and northwest endpoint of $B$. The closure of the tangle $A$ is the link obtained by connecting its northeast and northwest endpoints, southeast and southwest endpoints (see Figure\ref{fig:rational_tangle} (ii)).

For  a rational number $\frac{p}{q}$, take its continued fraction expression
\[ \frac{p}{q}= a_n + \cfrac{1}{a_{n-1} + \cfrac{1}{a_{n-2} + \cdots +\cfrac{1}{a_2+\cfrac{1}{a_1}}}}  \]
The rational tangle $Q(\frac{p}{q})$ is the tangle given by 
\[ Q\left(\frac{p}{q} \right) = \begin{cases}
\left( \cdots \left(\left(\left( \left[a_1 \right] \ast \left[\frac{1}{a_2} \right] \right) + [a_3 ]\right) \ast \cdots  \right)\ast\left[\frac{1}{a_{n-1}}\right]  \right) + [a_{n}] & (n:\mbox{odd}) \\
\left( \cdots \left(\left(\left( \left[\frac{1}{a_1} \right] + \left[a_2 \right] \right) \ast \left[\frac{1}{a_3} \right]\right)+  \cdots  \right)\ast\left[\frac{1}{a_{n-1}}\right]  \right) + [a_{n}] & (n:\mbox{even})
\end{cases}\]
(see Figure \ref{fig:rational_tangle} (iii)).

The isotopy class of the tangle $Q(\frac{p}{q})$ does not depend on a choice of a continued fraction.

\begin{definition}[(Length three) Montesinos knot]
The (length three) Montesinos knot $K(\frac{p_1}{q_1},\frac{p_2}{q_2},\frac{p_3}{q_3})$ is the closure of the tangle $Q(\frac{p_1}{q_1})+Q(\frac{p_2}{q_2})+Q(\frac{p_3}{q_3})$.
\end{definition}

For each permutation $\sigma$ of $\{1,2,3\}$, 
\[ K\left(\frac{p_1}{q_1},\frac{p_2}{q_2},\frac{p_3}{q_3}\right) = K\left(\frac{p_{\sigma(1)}}{q_{\sigma(1)}},\frac{p_{\sigma(2)}}{q_{\sigma(2)}},\frac{p_{\sigma(3)}}{q_{\sigma(3)}}\right).\]  
Furthermore,  for every $n \in \Z$. 
\[ K\left(\frac{p_1}{q_1},\frac{p_2}{q_2},\frac{p_3}{q_3}\right) = K\left(\frac{p_1}{q_1}+n,\frac{p_2}{q_2}-n,\frac{p_3}{q_3}\right). \]
It is known that $K(\frac{p_1}{q_1},\frac{p_2}{q_2},\frac{p_3}{q_3})= K(\frac{p'_1}{q'_1},\frac{p'_2}{q'_2},\frac{p'_3}{q'_3})$ if and only if they are related by these two operations. Namely,
\[ \left\{\frac{p_1}{q_1},\frac{p_2}{q_2},\frac{p_3}{q_3}\right\} = \left\{\frac{p'_1}{q'_1},\frac{p'_2}{q'_2},\frac{p'_3}{q'_3}\right\} \pmod 1 \ \mbox{ and } \frac{p_1}{q_1}+\frac{p_2}{q_2}+\frac{p_3}{q_3} = \frac{p'_1}{q'_1}+\frac{p'_2}{q'_2}+\frac{p'_3}{q'_3}.\]

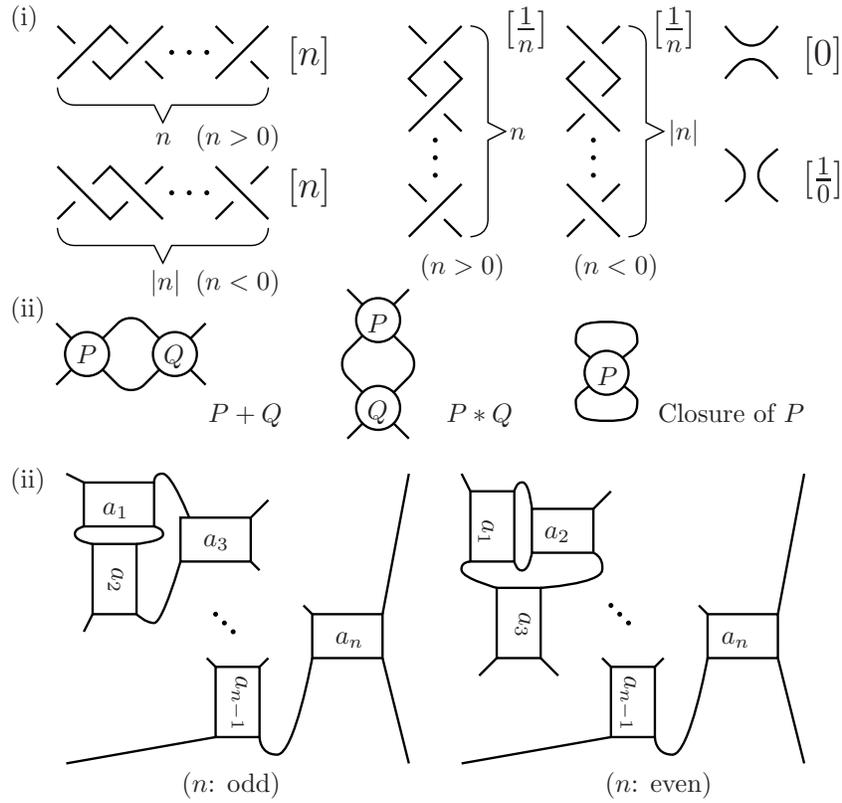
\begin{figure}[htbp]
\begin{center}
\includegraphics*[width=100mm]{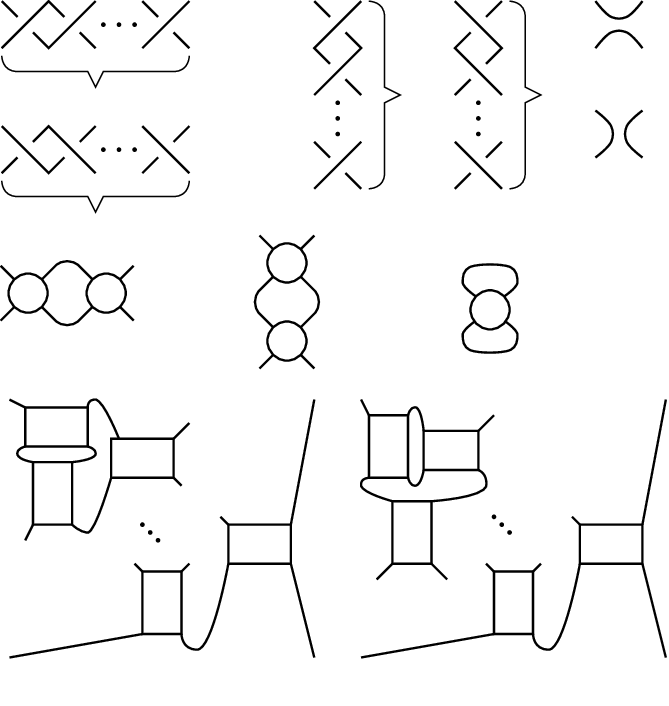}
\begin{picture}(0,0)
\put(-305,300){(i)}
\put(-305,190){(ii)}
\put(-305,125){(ii)}
\put(-200,285){\LARGE $[n]$}
\put(-200,235){\LARGE $[n]$}
\put(-235,255){$(n>0)$}
\put(-250,255){$n$}
\put(-235,200){$(n<0)$}
\put(-252,200){$|n|$}
\put(-118,295){\LARGE $[\frac{1}{n}]$}
\put(-62,295){\LARGE $[\frac{1}{n}]$}
\put(-116,257){$n$}
\put(-56,257){$|n|$}
\put(-280,172){$P$}
\put(-247,172){$Q$}
\put(-230,150){$P+Q$}
\put(-170,184){$P$}
\put(-170,151){$Q$}
\put(-140,150){$P \ast Q$}
\put(-83,164){$P$}
\put(-60,150){Closure of $P$}
\put(-5,285){\LARGE $[0]$}
\put(-5,238){\LARGE $[\frac{1}{0}]$}
\put(-150,206){$(n>0)$}
\put(-92,206){$(n<0)$}
\put(-270,115){$a_1$}
\put(-268,95){\rotatebox{270}{$a_2$}}
\put(-232,103){$a_3$}
\put(-222,54){\rotatebox{270}{$a_{n-1}$}}
\put(-182,66){$a_n$}
\put(-128,113){\rotatebox{270}{$a_1$}}
\put(-103,106){$a_2$}
\put(-114,78){\rotatebox{270}{$a_3$}}
\put(-75,54){\rotatebox{270}{$a_{n-1}$}}
\put(-36,66){$a_n$}
\put(-240,10){($n$: odd)}
\put(-80,10){($n$: even)}
\end{picture}
\caption{Rational tangle. (i) Definition of tangles $[n]$, $[\frac{1}{n}]$, $o, [\infty]=[\frac{1}{0}]$. (ii) Tangles $P+Q$, $P\ast Q$ and the closure of $P$. (iii) Rational tangle $Q(\frac{p}{q})$}
\label{fig:rational_tangle}
\end{center}
\end{figure}

\end{document}